\documentclass[11pt]{amsart}%
\usepackage{amsfonts}
\usepackage{amssymb}
\usepackage{amsthm,amsmath}
\usepackage{url}
\usepackage{amsmath}
\usepackage{graphicx}%
\providecommand{\U}[1]{\protect\rule{.1in}{.1in}}
\theoremstyle{plain}
\newtheorem{theorem}{Theorem}
\newtheorem{lemma}[theorem]{Lemma}
\newtheorem{proposition}[theorem]{Proposition}

\newtheorem*{interpolationproblem}{Interpolation Problem}

\theoremstyle{definition}

\newtheorem{example}[theorem]{Example}
\theoremstyle{remark}

\begin{document}
\title[Generalization of Goodstein's theorem]{A generalization of Goodstein's theorem: interpolation by polynomial functions
of distributive lattices}
\author{Miguel Couceiro}
\address[Miguel Couceiro]{Mathematics Research Unit, FSTC, University of Luxembourg \\
6, rue Coudenhove-Kalergi, L-1359 Luxembourg, Luxembourg}
\email{miguel.couceiro[at]uni.lu }
\author{Tam\'as Waldhauser}
\address[Tam\'as Waldhauser]{Mathematics Research Unit, FSTC, University of Luxembourg \\
6, rue Coudenhove-Kalergi, L-1359 Luxembourg, Luxembourg, and Bolyai
Institute, University of Szeged, Aradi v\'{e}rtan\'{u}k tere 1, H-6720 Szeged, Hungary}
\email{twaldha@math.u-szeged.hu}

\begin{abstract}
We consider the problem of interpolating functions partially defined over a
distributive lattice, by means of lattice polynomial functions. Goodstein's
theorem solves a particular instance of this interpolation problem on a
distributive lattice $L$ with least and greatest elements $0$ and $1$, resp.:
Given a function $f\colon\{0,1\}^{n}\to L$, there exists a lattice polynomial
function $p\colon L^{n}\to L$ such that $p|_{\{0,1\}^{n}}=f$ if an only if $f$
is monotone; in this case, the interpolating polynomial $p$ is unique.

We extend Goodstein's theorem to a wider class of partial functions $f\colon
D\to L$ over a distributive lattice $L$, not necessarily bounded, and where
$D\subseteq L^{n}$ is allowed to range over cuboids $D=\left\{  a_{1}%
,b_{1}\right\}  \times\cdots\times\left\{  a_{n},b_{n}\right\}  $ with
$a_{i},b_{i}\in L$ and $a_{i}<b_{i}$, and determine the class of such partial
functions which can be interpolated by lattice polynomial functions. In this
wider setting, interpolating polynomials are not necessarily unique; we
provide explicit descriptions of all possible lattice polynomial functions
which interpolate these partial functions, when such an interpolation is available.

\end{abstract}
\maketitle

\section{Introduction}

Let $L$ be a distributive lattice and let $f\colon D\rightarrow L~\left(
D\subseteq L^{n}\right)  $ be an $n$-ary partial function on $L$. In this
paper we are interested in the problem of extending such partial functions to
the whole domain $L^{n}$ by means of lattice polynomial functions, i.e.,
functions that can be represented as compositions of the lattice operations
$\wedge$ and $\vee$ and constants. More precisely, we aim at determining
necessary and sufficient conditions on the partial function $f$ that guarantee
the existence of a lattice polynomial function $p\colon L^{n}\rightarrow L$
which interpolates $f$, that is, $p|_{D}=f$.

An instance of this problem was considered by Goodstein \cite{Goo67} in the
case when $L$ is a bounded distributive lattice, and the functions to be
interpolated were of the form $f\colon\left\{  0,1\right\}  ^{n}\rightarrow
L$. Goodstein showed that such a function $f$ can be interpolated by lattice
polynomial functions if and only if it is monotone. Furthermore, if such an
interpolating polynomial function exists, then it is unique.

The general solution to the above mentioned interpolation problem eludes us.
However, we are able to generalize Goodstein's result by allowing $L$ to be an
arbitrary (possibly unbounded) distributive lattice and considering functions
$f\colon D\rightarrow L$, where $D=\left\{  a_{1},b_{1}\right\}  \times
\cdots\times\left\{  a_{n},b_{n}\right\}  $ with $a_{i},b_{i}\in L$ and
$a_{i}<b_{i}$. More precisely, we furnish necessary and sufficient conditions
for the existence of an interpolating polynomial function. As it will become
clear, in this more general setting, uniqueness is not guaranteed, and thus we
determine all possible interpolating polynomial functions.

The structure of the paper is as follows. In Section~\ref{sect preliminaries}
we recall basic background on polynomial functions over distributive lattices
(for general background see \cite{DavPri,Grae03}) and formalize the
interpolation problem that we are interested in. In Section~\ref{sect main} we
state and prove the characterization of those functions that can be
interpolated by polynomial functions and we describe the set of all solutions
of the interpolation problem. We discuss variations of the interpolation
problem in Section~\ref{sect concluding remarks} and relate our work to
earlier results obtained for finite chains in \cite{RG}. Finally, in
Section~\ref{sect decision} we consider potential applications of our results
in mathematical modeling of decision making.

\section{Preliminaries\label{sect preliminaries}}

Let $L$ be a bounded distributive lattice with least element $0$ and greatest
element $1$. It can be shown that a function $p\colon L^{n}\rightarrow L$ is a
lattice polynomial function if and only if there exist $c_{I}\in L$,
$I\subseteq\left[  n\right]  :=\left\{  1,\ldots,n\right\}  $, such that $p$
can be represented in \emph{disjunctive normal form} (DNF\ for short) by%
\begin{equation}
p\left(  \mathbf{x}\right)  =\bigvee_{I\subseteq\left[  n\right]  }%
\bigl(c_{I}\wedge\bigwedge_{i\in I}x_{i}\bigr),\,\text{ where $\mathbf{x}%
=\left(  x_{1},\ldots,x_{n}\right)  \in L^{n}$. } \label{eq DNF}%
\end{equation}
It is easy to verify that taking $c_{I}^{\prime}=\bigvee_{J\subseteq I}c_{J}
$, we also have%
\[
p\left(  \mathbf{x}\right)  =\bigvee_{I\subseteq\left[  n\right]  }%
\bigl(%
c_{I}^{\prime}\wedge\bigwedge_{i\in I}x_{i}%
\bigr)%
,
\]
and hence the coefficients $c_{I}$ can be assumed to be monotone in the sense
that $I\subseteq J$ implies $c_{I}\leq c_{J}$. This monotonicity assumption
allows us to recover the coefficients of the DNF from certain values of the
polynomial function $p$. Indeed, denoting by $\mathbf{1}_{I}$ the
characteristic vector of $I\subseteq\left[  n\right]  $ (i.e., the tuple
$\mathbf{1}_{I}\in L^{n}$ whose $i$-th component is $1$ if $i\in I$ and $0$ if
$i\notin I$), we then have that $p\left(  \mathbf{1}_{I}\right)  =c_{I}$. In
the sequel, we will always consider lattice polynomials in DNF, and we will
implicitly assume that the coefficients are monotone. These observations
contain the essence of Goodstein's theorem.

\begin{theorem}
[Goodstein \cite{Goo67}]\label{thm goodstein}Let $L$ be a bounded distributive
lattice, and let $f$ be a function $f\colon\left\{  0,1\right\}
^{n}\rightarrow L$. There exists a polynomial function $p$ over $L$ such that
$p|_{\left\{  0,1\right\}  ^{n}}=f $ if and only if $f$ is monotone. In this
case $p$ is uniquely determined, and can be represented by the DNF
\[
p\left(  \mathbf{x}\right)  =\bigvee_{I\subseteq\left[  n\right]  }
\bigl(f\left(  \mathbf{1}_{I}\right)  \wedge\bigwedge_{i\in I}x_{i}%
\bigr)\text{.}%
\]

\end{theorem}

Informally, Goodstein's theorem asserts that polynomial functions are uniquely
determined by their restrictions to the hypercube $\left\{  0,1\right\}  ^{n}%
$, and a function on the hypercube extends to a polynomial function if and
only if it is monotone.

Let us now consider a distributive lattice $L$ without least and greatest
elements. (We omit the analogous discussion of the cases where $L$ has one
boundary element.) Polynomial functions over $L$ can still be given in DNF of
the form (\ref{eq DNF}) by allowing the coefficients $c_{I}$ to take also the
values $0$ and $1$, which are considered as external boundary elements (see,
e.g., \cite{BCKLSz,CL}). For example, a polynomial function $p\left(
x,y\right)  =a\vee x\vee\left(  b\wedge x\wedge y\right)  $ can be rewritten
as $p\left(  x,y\right)  =a\vee\left(  1\wedge x\right)  \vee\left(  0\wedge
y\right)  \vee\left(  b\wedge x\wedge y\right)  $.

We can still assume monotonicity of the coefficients, and any such system
$c_{I}\in L\cup\left\{  0,1\right\}  \left(  I\subseteq\left[  n\right]
\right)  $ of coefficients gives rise to a polynomial function $p$ over $L$,
provided that $c_{\emptyset}\neq1$ and $c_{\left[  n\right]  }\neq0$. (The
latter two cases correspond to the constant $1$ and constant $0$ functions.)
Just like in the case of bounded lattices, there is a one-to-one
correspondence between such DNF's and polynomial functions, since we can
recover the coefficients of the DNF\ of $p$ from certain values of $p$. To see
this, let us choose elements $a<b$ from $L$ to play the role of $0$ and $1$,
and let $\mathbf{e}_{I}$ be the \textquotedblleft characteristic
vector\textquotedblright\ of $I\subseteq\left[  n\right]  $ (i.e., the tuple
$\mathbf{e}_{I}\in L^{n}$ whose $i$-th component is $b$ if $i\in I$ and $a$ if
$i\notin I$). If $a$ is sufficiently small (less than all non-zero
coefficients in the DNF of $p$) and $b$ is sufficiently large (greater than
all non-one coefficients in the DNF of $p$), then a routine computation shows
that%
\[
p\left(  \mathbf{e}_{I}\right)  =\left\{  \!\!%
\begin{array}
[c]{cc}%
c_{I} & \text{if~}c_{I}\in L,\\
a & \text{if }c_{I}=0,\\
b & \text{if }c_{I}=1.
\end{array}
\right.
\]

Thus we can learn the coefficient $c_{I}$ from the behavior of the value
$p\left(  \mathbf{e}_{I}\right)  $ by letting $a$ decrease and $b$ increase
indefinitely, i.e., the polynomial function $p$ is uniquely determined by its
values on a sufficiently large cube $\left\{  a,b\right\}  ^{n}$ (for a more
detailed discussion, see \cite{BCKLSz}). As the next example shows, this does
not imply that there is only one polynomial function that takes prescribed
values on a \emph{fixed} cube $\left\{  a,b\right\}  ^{n}$.

\begin{example}
\label{ex no least/greatest solution}Let $L$ be the lattice of open subsets of
a topological space $X$, and let $a,b\in L$ with $a\subset b$. Since $L$ is a
bounded distributive lattice, every unary polynomial function $p$ over $L$ can
be represented by a unique DNF of the form $p\left(  x\right)  =c_{0}%
\cup\left(  c_{1}\cap x\right)  $ with $c_{0},c_{1}\in L,c_{0}\subseteq c_{1}%
$. It is straightforward to verify that such a polynomial function satisfies
$p\left(  a\right)  =p\left(  b\right)  =b$ if and only if%
\[
b\setminus a\subseteq c_{0}\subseteq b\quad\text{and}\quad b\subseteq
c_{1}\subseteq X.
\]
Thus, there may be infinitely many polynomial functions $p$ whose restriction
to the \textquotedblleft one-dimensional cube\textquotedblright\ $\left\{
a,b\right\}  $ is constant $b$ (for instance, let $X$ be the real line, and
let $a$ and $b$ be open intervals).
\end{example}

Let us go one step further, and choose a \textquotedblleft
zero\textquotedblright\ and \textquotedblleft one\textquotedblright,\ possibly
different in each coordinate: Let $a_{i},b_{i}\in L$ with $a_{i}<b_{i}$ for
each $i\in\left[  n\right]  $, and let $\widehat{\mathbf{e}}_{I}$ be the
\textquotedblleft characteristic vector\textquotedblright\ of $I\subseteq
\left[  n\right]  $ (i.e., the tuple $\widehat{\mathbf{e}}_{I}\in L^{n}$ whose
$i$-th component is $b_{i}$ if $i\in I$ and $a_{i}$ if $i\notin I$). The task
of finding a polynomial function (or rather all polynomial functions) that
takes prescribed values on the tuples $\widehat{\mathbf{e}}_{I}$ can be
regarded as an interpolation problem.

\begin{interpolationproblem}
Given $D:=\left\{  \widehat{\mathbf{e}}_{I}:I\subseteq\left[  n\right]
\right\}  $ and $f\colon D \rightarrow L$, find all polynomial functions
$p\colon L^{n}\rightarrow L$ such that $p|_{D}=f$.
\end{interpolationproblem}

Note that here the function $f$ is given on the vertices of a rectangular box
(cuboid) instead of a cube as in Theorem~\ref{thm goodstein}. We will solve
this problem in Section~\ref{sect main}, thereby generalizing Goodstein's
theorem. Let us note that the problem can be interesting also in the case of
bounded lattices, for instance, if we do not have access to the values of the
polynomial function on $\left\{  0,1\right\}  ^{n}$, but only on some
\textquotedblleft internal\textquotedblright\ points. We will discuss such
applications in Section~\ref{sect decision}.

\section{Main results\label{sect main}}

In the sequel we assume that $D:=\left\{  \widehat{\mathbf{e}}_{I}\colon
I\subseteq\left[  n\right]  \right\}  $ and $f\colon D\rightarrow L$ are
given, and our goal is to find (the DNF of) all $n$-ary polynomial functions
$p$ over $L$ that satisfy $p|_{D}=f$. Clearly, monotonicity of $f$ is a
necessary condition for the existence of a solution of the Interpolation
Problem, but, in contrast with Goodstein's theorem, monotonicity is not always
sufficient in this more general setting. We will prove that the extra
condition that we need is the following:%
\begin{equation}
f\left(  \widehat{\mathbf{e}}_{I\cup\left\{  k\right\}  }\right)  \wedge
a_{k}\leq f\left(  \widehat{\mathbf{e}}_{I}\right)  \leq f\left(
\widehat{\mathbf{e}}_{I\setminus\left\{  k\right\}  }\right)  \vee
b_{k}\text{\qquad for all }I\subseteq\left[  n\right]  ,k\in\left[  n\right]
. \tag{$\star$}\label{eq star}%
\end{equation}
Observe that the first inequality is trivial if $k\in I$, and the second
inequality is trivial if $k\notin I$.

Our first lemma shows how to obtain inequalities between $f\left(
\widehat{\mathbf{e}}_{S}\right)  $ and $f\left(  \widehat{\mathbf{e}}%
_{T}\right)  $ for $S\subseteq T$ by repeated applications of (\ref{eq star}).

\begin{lemma}
\label{lemma iterated star}If the function $f$ satisfies
\upshape{(\ref{eq star})}, then for all $S\subseteq T\subseteq\left[
n\right]  $ we have%
\[
f\left(  \widehat{\mathbf{e}}_{T}\right)  \wedge\bigwedge\limits_{k\in
T\setminus S}a_{k}\leq f\left(  \widehat{\mathbf{e}}_{S}\right)  \text{ and
}f\left(  \widehat{\mathbf{e}}_{T}\right)  \leq f\left(  \widehat{\mathbf{e}%
}_{S}\right)  \vee\bigvee\limits_{k\in T\setminus S}b_{k}.
\]

\end{lemma}

\begin{proof}
We only prove the first inequality; the second one follows similarly. Let
$T\setminus S=\left\{  k_{1},\ldots,k_{r}\right\}  $, and let us apply (the
first inequality of) condition (\ref{eq star}) with $I=S\cup\left\{
k_{1},\ldots,k_{m-1}\right\}  $ and $k=k_{m}$ for $m=1,\ldots,r$:%
\begin{align*}
f\left(  \widehat{\mathbf{e}}_{S\cup\left\{  k_{1}\right\}  }\right)  \wedge
a_{k_{1}}  &  \leq f\left(  \widehat{\mathbf{e}}_{S}\right)  ,\\
f\left(  \widehat{\mathbf{e}}_{S\cup\left\{  k_{1},k_{2}\right\}  }\right)
\wedge a_{k_{2}}  &  \leq f\left(  \widehat{\mathbf{e}}_{S\cup\left\{
k_{1}\right\}  }\right)  ,\\
&  \;\;\vdots\\
f\left(  \widehat{\mathbf{e}}_{S\cup\left\{  k_{1},\ldots,k_{r}\right\}
}\right)  \wedge a_{k_{r}}  &  \leq f\left(  \widehat{\mathbf{e}}%
_{S\cup\left\{  k_{1},\ldots,k_{r-1}\right\}  }\right)  .
\end{align*}
Combining these $r$ inequalities, we get%
\[
f\left(  \widehat{\mathbf{e}}_{S\cup\left\{  k_{1},\ldots,k_{r}\right\}
}\right)  \wedge a_{k_{1}}\wedge\cdots\wedge a_{k_{r}}\leq f\left(
\widehat{\mathbf{e}}_{S}\right)  .\qedhere
\]

\end{proof}

Let us now show that (\ref{eq star}) is a necessary condition for the
existence of a solution of the Interpolation Problem.

\begin{lemma}
\label{lemma star is necessary}If there is a polynomial function $p$ over $L$
such that $p|_{D}=f$, then $f$ is monotone and satisfies \upshape{(\ref{eq star})}.
\end{lemma}

\begin{proof}
Assume that $p$ is a polynomial function that extends $f$. Since $p$ is
monotone, $f$ is also monotone. To show that (\ref{eq star}) holds, let us fix
$I\subseteq\left[  n\right]  $ and $k\in\left[  n\right]  $, and let us assume
that $k\notin I$ (the case $k\in I$ can be dealt with similarly). Let $\left(
\widehat{\mathbf{e}}_{I}\right)  _{k}^{x}\in L^{n}$ denote the $n$-tuple
obtained from $\widehat{\mathbf{e}}_{I}$ by replacing its $k$-th component by
the variable $x$. We can define a unary polynomial function $u$ over $L$ by
$u\left(  x\right)  :=p\left(  \left(  \widehat{\mathbf{e}}_{I}\right)
_{k}^{x}\right)  $. Using this notation, (\ref{eq star}) takes the form
$u\left(  b_{k}\right)  \wedge a_{k}\leq u\left(  a_{k}\right)  $. The DNF of
$u$ is of the form $u\left(  x\right)  =c_{0}\vee\left(  c_{1}\wedge x\right)
$, where $c_{0},c_{1}\in L\cup\left\{  0,1\right\}  $. Using distributivity
and the fact that $a_{k}<b_{k}$, we can now easily prove the desired
inequality:%
\begin{multline*}
u\left(  b_{k}\right)  \wedge a_{k}=\left(  c_{0}\vee\left(  c_{1}\wedge
b_{k}\right)  \right)  \wedge a_{k}=\left(  c_{0}\wedge a_{k}\right)
\vee\left(  c_{1}\wedge b_{k}\wedge a_{k}\right) \\
=\left(  c_{0}\wedge a_{k}\right)  \vee\left(  c_{1}\wedge a_{k}\right)  \leq
c_{0}\vee\left(  c_{1}\wedge a_{k}\right)  =u\left(  a_{k}\right)  .\qedhere
\end{multline*}

\end{proof}

To find all polynomial functions $p$ satisfying $p|_{D}=f$, we will make use
of the Birkhoff-Priestley representation theorem to embed $L$ into a Boolean
algebra $B$.\ For the sake of canonicity, we assume that $L$ generates $B$;
under this assumption $B$ is uniquely determined up to isomorphism. The
boundary elements of $B$ will be denoted by $0$ and $1$. This notation will
not lead to ambiguity since if $L$ has a least (resp. greatest) element, then
it must coincide with $0$ (resp. $1$). The complement of an element $a\in B$
is denoted by $a^{\prime}$. Given a function $f\colon D\rightarrow L$, we
define the following two elements in $B$ for each $I\subseteq\left[  n\right]
$:%
\[
c_{I}^{-}:=f\left(  \widehat{\mathbf{e}}_{I}\right)  \wedge\bigwedge
\limits_{i\notin I}a_{i}^{\prime},\qquad c_{I}^{+}:=f\left(  \widehat
{\mathbf{e}}_{I}\right)  \vee\bigvee\limits_{i\in I}b_{i}^{\prime}.
\]
Observe that $c_{I}^{-}\leq c_{I}^{+}$, and if $f$ is monotone, then
$I\subseteq J$ implies $c_{I}^{-}\leq c_{J}^{-}$ and $c_{I}^{+}\leq c_{J}^{+}%
$. Let $p^{-}$ and $p^{+}$ be the polynomial functions over $B$ which are
given by these two systems of coefficients. We will see that $p^{-}$ and
$p^{+}$ are the least and greatest polynomial functions over $B$ whose
restriction to $D$ coincides with $f$ (whenever there exists such a polynomial function).

\begin{lemma}
\label{lemma estimating p+}If $f$ is monotone and satisfies
\upshape{(\ref{eq star})}, then $p^{+}\left(  \widehat{\mathbf{e}}_{J}\right)
\leq f\left(  \widehat{\mathbf{e}}_{J}\right)  $ for all $J\subseteq\left[
n\right]  $.
\end{lemma}

\begin{proof}
Let us fix $J\subseteq\left[  n\right]  $ and consider the value of $p^{+}$ at
$\widehat{\mathbf{e}}_{J}$:%
\[
p^{+}\left(  \widehat{\mathbf{e}}_{J}\right)  =\bigvee_{I\subseteq\left[
n\right]  }%
\bigl(%
c_{I}^{+}\wedge\bigwedge_{j\in I}\left(  \widehat{\mathbf{e}}_{J}\right)  _{j}%
\bigr)%
=\bigvee_{I\subseteq\left[  n\right]  }%
\bigl(%
c_{I}^{+}\wedge\bigwedge_{j\in I\setminus J}a_{j}\wedge\bigwedge_{j\in I\cap
J}b_{j}%
\bigr)%
.
\]
It is sufficient to verify that each joinand is at most $f\left(
\widehat{\mathbf{e}}_{J}\right)  $. Taking into account the definition of
$c_{I}^{+}$, this amounts to showing that%
\begin{equation}%
\bigl(%
f\left(  \widehat{\mathbf{e}}_{I}\right)  \vee\bigvee\limits_{i\in I}%
b_{i}^{\prime}%
\bigr)%
\wedge\bigwedge_{j\in I\setminus J}a_{j}\wedge\bigwedge_{j\in I\cap J}%
b_{j}\leq f\left(  \widehat{\mathbf{e}}_{J}\right)
\label{eq inequality to prove}%
\end{equation}
holds for all $I\subseteq\left[  n\right]  $. Distributing meets over joins,
the left hand side of (\ref{eq inequality to prove}) becomes%
\begin{equation}%
\bigl(%
f\left(  \widehat{\mathbf{e}}_{I}\right)  \wedge\bigwedge_{j\in I\setminus
J}a_{j}\wedge\bigwedge_{j\in I\cap J}b_{j}%
\bigr)%
\vee\bigvee\limits_{i\in I}%
\bigl(%
b_{i}^{\prime}\wedge\bigwedge_{j\in I\setminus J}a_{j}\wedge\bigwedge_{j\in
I\cap J}b_{j}%
\bigr)%
. \label{eq join to estimate}%
\end{equation}
Let us examine each joinand of this expression. For each $i\in I$, the joinand
involving $b_{i}^{\prime}$ equals $0$, since%
\[
b_{i}^{\prime}\wedge\bigwedge_{j\in I\setminus J}a_{j}\wedge\bigwedge_{j\in
I\cap J}b_{j}\leq b_{i}^{\prime}\wedge\bigwedge_{j\in I\setminus J}b_{j}%
\wedge\bigwedge_{j\in I\cap J}b_{j}=b_{i}^{\prime}\wedge\bigwedge_{j\in
I}b_{j}\leq b_{i}^{\prime}\wedge b_{i}=0.
\]

The joinand of (\ref{eq join to estimate}) that involves $f\left(
\widehat{\mathbf{e}}_{I}\right)  $ can be estimated using (\ref{eq star}) and
Lemma~\ref{lemma iterated star} (with $T=I$ and $S=I\cap J$):%
\[
f\left(  \widehat{\mathbf{e}}_{I}\right)  \wedge\bigwedge_{j\in I\setminus
J}a_{j}\wedge\bigwedge_{j\in I\cap J}b_{j}\leq f\left(  \widehat{\mathbf{e}%
}_{I}\right)  \wedge\bigwedge_{j\in I\setminus\left(  I\cap J\right)  }%
a_{j}\leq f\left(  \widehat{\mathbf{e}}_{I\cap J}\right)  .
\]
Since $f$ is monotone, we have $f\left(  \widehat{\mathbf{e}}_{I\cap
J}\right)  \leq f\left(  \widehat{\mathbf{e}}_{J}\right)  $, and this proves
(\ref{eq inequality to prove}).
\end{proof}

The following lemma is the dual of Lemma~\ref{lemma estimating p+}, and it can
be proved by using the conjunctive normal form of $p^{-}$.

\begin{lemma}
\label{lemma estimating p-}If $f$ is monotone and satisfies
\upshape{(\ref{eq star})}, then $p^{-}\left(  \widehat{\mathbf{e}}_{J}\right)
\geq f\left(  \widehat{\mathbf{e}}_{J}\right)  $ for all $J\subseteq\left[
n\right]  $.
\end{lemma}

The estimates obtained in the previous two lemmas allow us to find all
solutions of our interpolation problem over $B$, whenever a solution exists.

\begin{theorem}
\label{thm all solutions over B}Let $D=\left\{  \widehat{\mathbf{e}}_{I}\colon
I\subseteq\left[  n\right]  \right\}  $ and $f\colon D\rightarrow L$ be given,
as in the Interpolation Problem. Suppose that $f$ is monotone and satisfies
{\upshape{(\ref{eq star})}}, and let $p$ be an $n$-ary polynomial function
over $B$ given by the DNF corresponding to a system of coefficients $c_{I}\in
B\left(  I\subseteq\left[  n\right]  \right)  $. Then the following three
conditions are equivalent:%
\renewcommand{\labelenumi}{\textup{(\roman{enumi})}}
\renewcommand{\theenumi}{\labelenumi}%

\begin{enumerate}
\item \label{item solution in thm all solutions over B}$p|_{D}=f$;

\item \label{item coefficientwise in thm all solutions over B}for all
$I\subseteq\left[  n\right]  $ the inequalities $c_{I}^{-}\leq c_{I}\leq
c_{I}^{+}$ hold;

\item \label{item pointwise in thm all solutions over B}for all $\mathbf{x}\in
L^{n}$ we have $p^{-}\left(  \mathbf{x}\right)  \leq p\left(  \mathbf{x}%
\right)  \leq p^{+}\left(  \mathbf{x}\right)  $.
\end{enumerate}
\end{theorem}

\begin{proof}
Implication \ref{item coefficientwise in thm all solutions over B}$\implies
$\ref{item pointwise in thm all solutions over B} is trivial. To prove
\ref{item solution in thm all solutions over B}$\implies$%
\ref{item coefficientwise in thm all solutions over B}, assume that $p|_{D}%
=f$, i.e., $p\left(  \widehat{\mathbf{e}}_{J}\right)  =f\left(  \widehat
{\mathbf{e}}_{J}\right)  $ for all $J\subseteq\left[  n\right]  $. Then we can
replace $f\left(  \widehat{\mathbf{e}}_{J}\right)  $ by $p\left(
\widehat{\mathbf{e}}_{J}\right)  $ in the definition of $c_{J}^{-}$, and we
can compute its value by substituting $\widehat{\mathbf{e}}_{J}$ into the
DNF\ of $p$:%
\begin{align*}
c_{J}^{-}  &  =f\left(  \widehat{\mathbf{e}}_{J}\right)  \wedge\bigwedge
\limits_{j\notin J}a_{j}^{\prime}=p\left(  \widehat{\mathbf{e}}_{J}\right)
\wedge\bigwedge\limits_{j\notin J}a_{j}^{\prime}=%
\Bigl(%
\bigvee_{I\subseteq\left[  n\right]  }%
\bigl(%
c_{I}\wedge\bigwedge_{i\in I}\left(  \widehat{\mathbf{e}}_{J}\right)  _{i}%
\bigr)%
\Bigr)%
\wedge\bigwedge\limits_{j\notin J}a_{j}^{\prime}\\
&  =\bigvee_{I\subseteq\left[  n\right]  }%
\bigl(%
c_{I}\wedge\bigwedge_{i\in I\setminus J}a_{i}\wedge\bigwedge_{i\in I\cap
J}b_{i}\wedge\bigwedge\limits_{j\notin J}a_{j}^{\prime}%
\bigr)%
.
\end{align*}
If there exists $i\in I\setminus J$, then $a_{i}\wedge a_{i}^{\prime}=0$
appears in the joinand corresponding to $I$, hence we can omit each of these
terms from the join, and keep only those where $I\setminus J=\emptyset$:%
\[
c_{J}^{-}=\bigvee_{I\subseteq J}%
\bigl(%
c_{I}\wedge\bigwedge_{i\in I\setminus J}a_{i}\wedge\bigwedge_{i\in I\cap
J}b_{i}\wedge\bigwedge\limits_{j\notin J}a_{j}^{\prime}%
\bigr)%
\leq\bigvee_{I\subseteq J}c_{I}=c_{J}.
\]
This proves $c_{J}^{-}\leq c_{J}$. The inequality $c_{J}\leq c_{J}^{+}$ can be
proved by a dual argument.

Finally, to prove \ref{item pointwise in thm all solutions over B}$\implies
$\ref{item solution in thm all solutions over B}, let us assume that
$p^{-}\leq p\leq p^{+}$ holds in the pointwise ordering of functions. Applying
Lemma~\ref{lemma estimating p+} and Lemma~\ref{lemma estimating p-}, we get
the following chain of inequalities for every $I\subseteq\left[  n\right]  $:%
\[
f\left(  \widehat{\mathbf{e}}_{I}\right)  \leq p^{-}\left(  \widehat
{\mathbf{e}}_{I}\right)  \leq p\left(  \widehat{\mathbf{e}}_{I}\right)  \leq
p^{+}\left(  \widehat{\mathbf{e}}_{I}\right)  \leq f\left(  \widehat
{\mathbf{e}}_{I}\right)  .
\]
This implies $p\left(  \widehat{\mathbf{e}}_{I}\right)  =f\left(
\widehat{\mathbf{e}}_{I}\right)  $ for all $I\subseteq\left[  n\right]  $,
therefore we have $p|_{D}=f$.
\end{proof}

Note that in Lemma~\ref{lemma star is necessary} we did not make use of the
fact that $p$ is a polynomial function over $L$: the proof works also for
polynomial functions over $B$. This fact together with
Theorem~\ref{thm all solutions over B} shows that monotonicity and property
(\ref{eq star}) of $f$ are necessary and sufficient for the existence of a
solution of our interpolation problem over $B$. This observation leads to the
following result.

\begin{theorem}
\label{thm all solutions over L}The Interpolation Problem has a solution if
and only if $f$ is monotone and satisfies {\upshape{(\ref{eq star})}}. In this
case a polynomial function $p$ over $L$ verifies $p|_{D}=f$ if and only if
$c_{I}^{-}\leq c_{I}\leq c_{I}^{+}$ holds for the coefficients $c_{I}$ of the
DNF of $p$ for all $I\subseteq\left[  n\right]  $. In particular, $p$ can be
chosen as the polynomial function $p_{0}$ given by the coefficients
$c_{I}=f\left(  \widehat{\mathbf{e}}_{I}\right)  $:%
\[
p_{0}\left(  \mathbf{x}\right)  =\bigvee_{I\subseteq\left[  n\right]  }%
\bigl(%
f\left(  \widehat{\mathbf{e}}_{I}\right)  \wedge\bigwedge_{i\in I}x_{i}%
\bigr)%
\text{.}%
\]

\end{theorem}

\begin{proof}
The necessity of the conditions has been established in
Lemma~\ref{lemma star is necessary}. To prove the sufficiency, we just need to
observe that if $f$ is monotone and satisfies (\ref{eq star}), then the
polynomial function $p_{0}$ is a solution of the Interpolation Problem by
Theorem~\ref{thm all solutions over B}, as $c_{I}^{-}\leq f\left(
\widehat{\mathbf{e}}_{I}\right)  \leq c_{I}^{+}$ follows immediately from the
definition of $c_{I}^{-}$ and $c_{I}^{+}$. Since $f\left(  \widehat
{\mathbf{e}}_{I}\right)  \in L$ for all $I\subseteq\left[  n\right]  $, the
polynomial function $p_{0}$ is actually a polynomial function over $L$. The
description of the set of all solutions over $L$ also follows from
Theorem~\ref{thm all solutions over B}.
\end{proof}

Let us note that if $L$ is bounded and $a_{i}=0,b_{i}=1$ for all $i\in\left[
n\right]  $, then Theorem~\ref{thm all solutions over L} reduces to
Goodstein's theorem. Indeed, in this case (\ref{eq star}) holds trivially,
hence a solution exists if and only if $f$ is monotone. Moreover, we have
$c_{I}^{-}=c_{I}^{+}=f\left(  \widehat{\mathbf{e}}_{I}\right)  $, hence
$p_{0}$ (which is the same as the polynomial function given in
Theorem~\ref{thm goodstein}) is the only solution of the Interpolation Problem.

\section{Variations\label{sect concluding remarks}}

We have seen that monotonicity and property (\ref{eq star}) are necessary and
sufficient to guarantee the existence of a solution of the Interpolation
Problem. The following example shows that these two conditions are
independent, hence neither of them can be dropped.

\begin{example}
Let $L$ be a distributive lattice, let $a,b,c\in L$ such that $a<b<c$, and let
$D=\left\{  a,b\right\}  $. Then the function $f\colon D\rightarrow L$ defined
by $f\left(  a\right)  =b$, $f\left(  b\right)  =a$ satisfies (\ref{eq star})
but it is not monotone, while the function $g\colon D\rightarrow L$ defined by
$g\left(  a\right)  =a$, $g\left(  b\right)  =c$ is monotone but it does not
satisfy (\ref{eq star}).
\end{example}

Considering polynomial functions over the Boolean algebra $B$ generated by
$L$, the Interpolation Problem has a least and a greatest solution, namely
$p^{-}$ and $p^{+}$, whenever a solution exists (see
Theorem~\ref{thm all solutions over B}). On the other hand, the instance of
the Interpolation Problem considered in
Example~\ref{ex no least/greatest solution} has no least solution over $L$
itself (since usually there is no least open set containing $b\setminus a$),
and a dual example shows that in general there is no greatest solution over
$L$. However, if $L$ is complete, then extremal solutions exist over $L$. To
describe these, let us introduce the following notation. For an arbitrary
$b\in B$, we define the elements $\mathrm{cl}\left(  b\right)  $ and
$\mathrm{int}\left(  b\right)  $ of $L$ by%
\[
\mathrm{cl}\left(  b\right)  :=\bigwedge_{\substack{a\in L\\a\geq b}%
}a\quad\text{and}\quad\mathrm{int}\left(  b\right)  :=\bigvee_{\substack{a\in
L\\a\leq b}}a.
\]
Completeness of $L$ ensures that these (possibly infinite) meets and joins
exist, and one can verify that $\mathrm{cl}$ is a closure operator on $B$ (the
closed elements being exactly the elements of $L$), while $\mathrm{int}$ is
the dual closure operator on $B$ (also called as \textquotedblleft interior
operator\textquotedblright).

\begin{theorem}
\label{thm complete lattices}If $L$ is a complete distributive lattice, then a
polynomial function $p$ over $L$ is a solution of the Interpolation Problem if
and only if $\mathrm{cl}\left(  c_{I}^{-}\right)  \leq c_{I}\leq
\mathrm{int}\left(  c_{I}^{+}\right)  $ holds for the coefficients $c_{I}$ of
the DNF of $p$, for all $I\subseteq\left[  n\right]  $.
\end{theorem}

\begin{proof}
Theorem~\ref{thm complete lattices} follows directly from
Theorem~\ref{thm all solutions over L}, since, by the very definition of
$\mathrm{cl}$ and $\mathrm{int}$, we have that $c_{I}^{-}\leq c_{I}\leq
c_{I}^{+}$ holds for a given $c_{I}\in L$ if and only if $\mathrm{cl}\left(
c_{I}^{-}\right)  \leq c_{I}\leq\mathrm{int}\left(  c_{I}^{+}\right)  $.
\end{proof}

Now let us consider a general version of the Interpolation Problem, where $D$
is an arbitrary subset of $L^{n}$, not necessarily the set of vertices of a
rectangular box. This problem is still open for distributive lattices;
however, for finite chains the solution has been given in \cite{RG}. That
paper deals with Sugeno integrals (cf. Section~\ref{sect decision}) instead of
lattice polynomials; here we reformulate the criterion for the existence of a
solution (Theorem~3 in \cite{RG}) in the language of lattice theory.

\begin{theorem}
[\cite{RG}]\label{thm Rico-Grabisch}Let $L$ be a finite chain, and let $D$ be
an arbitrary subset of $L^{n}$. A function $f\colon D\rightarrow L$ extends to
a lattice polynomial function on $L$ if and only if%
\begin{equation}
\forall\mathbf{a},\mathbf{b}\in D:~f\left(  \mathbf{a}\right)  <f\left(
\mathbf{b}\right)  \implies\exists i\in\left[  n\right]  :a_{i}\leq f\left(
\mathbf{a}\right)  <f\left(  \mathbf{b}\right)  \leq b_{i}. \label{eq RG}%
\end{equation}

\end{theorem}

Let us explore the relationship between Theorem~\ref{thm Rico-Grabisch} and
Theorem~\ref{thm all solutions over L}. Our condition (\ref{eq star}) is
defined only for sets $D$ of the form $D=\left\{  \widehat{\mathbf{e}}%
_{I}:I\subseteq\left[  n\right]  \right\}  $, whereas (\ref{eq RG}) can be
interpreted for any set $D\subseteq L^{n}$ for any distributive lattice $L$.
Hence it is natural to ask whether Theorem~\ref{thm Rico-Grabisch} remains
valid for arbitrary distributive lattices. As the following example shows, if
$L$ is not a chain, then it can be the case that (\ref{eq RG}) is neither
sufficient nor necessary for the existence of a solution of the Interpolation
Problem, not even for the special kind of sets $D$ that we considered in this paper.

\begin{example}
\label{ex RG neither sufficient nor necessary}Let $L=\{0,1,a,b\}$ with
$0<a,b<1$ and $a,b$ incomparable. Let $n=1$ and $D=\{0,b\}$, and define
$f\colon D\rightarrow L$ by $f(0)=b$, $f(b)=a$ and $g\colon D\rightarrow L$ by
$g(0)=a$, $g(b)=1$ Then $f$ trivially satisfies (\ref{eq RG}), but $f$ is not
monotone, hence it is not the restriction of any polynomial function. On the
other hand, $g$ does not satisfy (\ref{eq RG}), although it is the restriction
of the polynomial function $p\left(  x\right)  =x\vee a$ to $D$.
\end{example}

Observe that if $L$ is a chain, then (\ref{eq RG}) implies that $f$ is
monotone\footnote{Of course, this follows from Theorem~\ref{thm Rico-Grabisch}%
, but it is also easy to verify directly.}, but this is not true for arbitrary
distributive lattices (see the example above). Thus we may want to require
that $f$ is a monotone function satisfying (\ref{eq RG}). We will prove below
that if $D$ is of \textquotedblleft rectangular\textquotedblright\ shape, then
monotonicity of $f$ and condition (\ref{eq RG}) are sufficient to ensure that
$f$ extends to a polynomial function (but (\ref{eq RG}) is not necessary, as
we have seen in Example~\ref{ex RG neither sufficient nor necessary}).

\begin{proposition}
\label{prop RG and monotone imply star}Let $L$ be a distributive lattice and
$D=\left\{  \widehat{\mathbf{e}}_{I}\colon I\subseteq\left[  n\right]
\right\}  $ as in the Interpolation\ Problem. If $f\colon D\rightarrow L$ is
monotone and satisfies \textup{(\ref{eq RG})}, then there exists a polynomial
function $p$ over $L$ such that $p|_{D}=f$.
\end{proposition}

\begin{proof}
Let $f\colon D\rightarrow L$ be a monotone function satisfying (\ref{eq RG}).
By Theorem~\ref{thm all solutions over L}, we only have to prove that $f$ also
satisfies (\ref{eq star}). Let us assume that $k\notin I$; the other case is
similar. Then only $f\left(  \widehat{\mathbf{e}}_{I\cup\left\{  k\right\}
}\right)  \wedge a_{k}\leq f\left(  \widehat{\mathbf{e}}_{I}\right)  $ needs
to be verified, as the second inequality of (\ref{eq star}) is trivial in this
case. Since $f$ is monotone, we have $f\left(  \widehat{\mathbf{e}}%
_{I}\right)  \leq f\left(  \widehat{\mathbf{e}}_{I\cup\left\{  k\right\}
}\right)  $, and if equality holds here, then we are done. On the other hand,
if $f\left(  \widehat{\mathbf{e}}_{I}\right)  <f\left(  \widehat{\mathbf{e}%
}_{I\cup\left\{  k\right\}  }\right)  $, then (\ref{eq RG}) implies that there
is an $i\in\left[  n\right]  $ such that%
\begin{equation}
\left(  \widehat{\mathbf{e}}_{I}\right)  _{i}\leq f\left(  \widehat
{\mathbf{e}}_{I}\right)  <f\left(  \widehat{\mathbf{e}}_{I\cup\left\{
k\right\}  }\right)  \leq\left(  \widehat{\mathbf{e}}_{I\cup\left\{
k\right\}  }\right)  _{i}. \label{eq RG and monotone implies star}%
\end{equation}
This is clearly impossible for $i\neq k$, since then the $i$-th component of
$\widehat{\mathbf{e}}_{I}$ and $\widehat{\mathbf{e}}_{I\cup\left\{  k\right\}
}$ is the same (namely, $a_{i}$). Thus we must have $i=k$, and then
(\ref{eq RG and monotone implies star}) reads as%
\[
a_{k}\leq f\left(  \widehat{\mathbf{e}}_{I}\right)  <f\left(  \widehat
{\mathbf{e}}_{I\cup\left\{  k\right\}  }\right)  \leq b_{k}.
\]
From this we immediately obtain the desired inequality:%
\[
f\left(  \widehat{\mathbf{e}}_{I\cup\left\{  k\right\}  }\right)  \wedge
a_{k}\leq a_{k}\leq f\left(  \widehat{\mathbf{e}}_{I}\right)  . \qedhere
\]

\end{proof}

Finally, we give an example that shows that monotonicity and condition
(\ref{eq RG}) together do not guarantee the existence of a solution of the
Interpolation Problem if $L$ is an arbitrary distributive lattice and $D$ is
an arbitrary subset of $L^{n}$. Thus it remains as a topic of further research
to find an appropriate criterion for the existence of an interpolating lattice
polynomial function in this general setting.

\begin{example}
\label{ex RG and monotone not sufficient}Let $L$ be the same lattice as in
Example~\ref{ex RG neither sufficient nor necessary}, and let $D=\left\{
a,b\right\}  $. Then the function $f\colon D\rightarrow L$ defined by
$f(a)=b$, $f(b)=a$ is monotone and satisfies (\ref{eq RG}), but it is not the
restriction of a polynomial function.
\end{example}

\section{Application in decision making\label{sect decision}}

The original motivation for considering the Interpolation Problem lies in the
following mathematical model of multicriteria decision making. Let us assume
that we have a set of alternatives from which we would like to choose the best
one (e.g., a house to buy). Several properties of these alternatives could be
important in making the decision (e.g., the size, price, etc., of a house),
and this very fact can make the decision difficult (for instance, maybe it is
not clear whether a cheap and small house is better than a big and expensive
one). To overcome this difficulty, the values corresponding to the various
properties of each alternative should be combined to a single value, which can
then be easily compared.

To formalize this situation, let us assume that there are $n$ criteria along
which the alternatives are evaluated, and these take their values in linearly
ordered sets $L_{1},\ldots,L_{n}$. These linearly ordered sets could be
quantitative scales (e.g., $L_{1}$ could be the real interval $\left[
40,200\right]  $, measuring the size of a house in square meters) or
qualitative scales (e.g., $L_{1}$ could be the finite chain $\left\{
\text{very small}<\text{small}<\text{big}<\text{very big}\right\}  $). Thus,
to each alternative corresponds a profile $\mathbf{x}\in L_{1}\times
\cdots\times L_{n}$.\ Since this product is usually not a linearly ordered
set, some alternatives may be incomparable. Therefore, we choose a common
scale $L$, and monotone functions $\varphi_{i}\colon L_{i}\rightarrow L\left(
i\in\left[  n\right]  \right)  $ to translate the values corresponding to the
different criteria (which may have different units of measure, e.g., square
meters, euros, etc.) to this common scale, and which are then combined into a
single value (for each alternative) by a so-called aggregation function
$p\colon L^{n}\rightarrow L$. In this way we obtain a function $U\colon
L_{1}\times\cdots\times L_{n}\rightarrow L$ defined by%
\begin{equation}
U\left(  \mathbf{x}\right)  =p\left(  \varphi_{1}\left(  x_{1}\right)
,\ldots,\varphi_{n}\left(  x_{n}\right)  \right)  ,
\label{eq pseudo-polynomial}%
\end{equation}
and we can choose the alternative that maximizes $U$. The function $U$ is
called a global utility function, whereas the maps $\varphi_{i}$ are called
local utility functions. The relevance of such functions is attested by their
many applications in decision making, in particular, in representing
preference relations \cite{BouDubPraPir09}.

It is common to choose the real interval $\left[  0,1\right]  $ for $L$, and
consider $\varphi_{i}\left(  x_{i}\right)  $ as a kind of \textquotedblleft
score\textquotedblright\ with respect to the $i$-th criterion. In this case,
simple aggregation functions $p$ are for instance the weighted arithmetic
means, but there are of course other, more elaborate ways of aggregating the
scores such as the so-called Choquet integrals. However, in the qualitative
approach, where only the ordering between scores is taken into account (for
instance, when $L=\left\{  \text{bad}<\text{OK}<\text{good}<\text{excellent}%
\right\}  $), such operators are of little use since they rely heavily on the
arithmetic structure of the real unit interval. In the latter setting, one of
the most prominent class of aggregation functions is that of discrete Sugeno
integrals, which coincides with the class of idempotent lattice polynomial
functions (see \cite{Mar09}).

In \cite{CW3} and \cite{Belfast} a more general situation was considered: $L$
is an arbitrary finite distributive lattice, the lattice polynomial functions
are not assumed to be idempotent, and the local utility functions are not
assumed to be monotone (instead they have to satisfy the boundary conditions
$\varphi_{i}\left(  0_{i}\right)  \leq\varphi_{i}\left(  x_{i}\right)
\leq\varphi_{i}\left(  1_{i}\right)  $ for all $x_{i}\in L_{i}$, where $0_{i}$
and $1_{i}$ denote the least and greatest element of $L_{i}$). The
corresponding compositions (\ref{eq pseudo-polynomial}) were called
pseudo-polynomial functions, and several axiomatizations were given for this
class of functions. Besides axiomatization, another noteworthy problem is the
factorization of such functions: given a function $U\colon L_{1}\times
\cdots\times L_{n}\rightarrow L$, find all factorizations of $U$ in the form
(\ref{eq pseudo-polynomial}). Such a factorization can be useful in real-life
applications, when only the function $U$ is available (from empirical
observations), and an analysis of the behavior of the local utility functions
$\varphi_{i}$ and of the aggregation function $p$ could give valuable
information about the decision maker's attitude.

Suppose that we have already found the local utility functions $\varphi_{i}$
(see \cite{CW3} and \cite{Belfast} for a method to find them), and let
$a_{i}=\varphi_{i}\left(  0_{i}\right)  ,b_{i}=\varphi_{i}\left(
1_{i}\right)  $. If $\mathbf{x}\in L_{1}\times\cdots\times L_{n}$ is such that
$x_{i}=1_{i}$ if $i\in I$ and $x_{i}=0_{i}$ if $i\notin I$, then $U\left(
\mathbf{x}\right)  =p\left(  \widehat{\mathbf{e}}_{I}\right)  $. Thus, knowing
the global utility function $U$, we have information about $p|_{D}$, and we
can use Theorem~\ref{thm all solutions over L} to find all possible lattice
polynomial functions $p$ that can appear in a factorization
(\ref{eq pseudo-polynomial}) of $U$. (Of course, one has to take into account
the other values of $U$ as well, but this can be done by using the boundary conditions.)

\subsubsection*{Acknowledgments.}

The first named author is supported by the internal research project
F1R-MTH-PUL-09MRDO of the University of Luxembourg. The second named author
acknowledges that the present project is supported by the
\hbox{T\'{A}MOP-4.2.1/B-09/1/KONV-2010-0005} program of the National
Development Agency of Hungary, by the Hungarian National Foundation for
Scientific Research under grants no.\ K77409 and K83219, by the National
Research Fund of Luxembourg, and cofunded under the Marie Curie Actions of the
European Commission \hbox{(FP7-COFUND).}


\bigskip

\begin{thebibliography}{99}                                                                                               %


\bibitem {BCKLSz}Behrisch, M., Couceiro, M., Kearnes, K., Lehtonen, E.,
Szendrei, \'{A}.: Commuting polynomial operations of distributive lattices. To
appear in Order

\bibitem {BouDubPraPir09}Bouyssou, D., Dubois, D., Prade, H., Pirlot, M.
(eds): Decision-Making Process -- Concepts and Methods. ISTE/John Wiley (2009)

\bibitem {CL}Couceiro, M., Lehtonen, E.: Self-commuting lattice polynomial
functions on chains. Aequationes Math. 81(3), 263--278 (2011)

\bibitem {CM}Couceiro, M., Marichal, J.-L.: Characterizations of discrete
Sugeno integrals as polynomial functions over distributive lattices. Fuzzy
Sets and Systems 161(5), 694--707 (2010)

\bibitem {CW3}Couceiro, M., Waldhauser, T.: Axiomatizations and factorizations
of Sugeno utility functions. Internat. J. Uncertain. Fuzziness Knowledge-Based
Systems 19(4), 635--658 (2011)

\bibitem {Belfast}Couceiro, M., Waldhauser, T.: Pseudo-polynomial functions
over finite distributive lattices. In: Liu, W. (ed.) ECSQARU 2011. LNCS
(LNAI), vol. 6717, pp. 545--556. Springer (2011)

\bibitem {DavPri}Davey, B. A., Priestley, H. A.: Introduction to Lattices and
Order. Cambridge University Press, New York (2002)

\bibitem {Goo67}Goodstein, R.~L.: The Solution of Equations in a Lattice.
Proc. Roy. Soc. Edinburgh Section A, 67, 231--242 (1965/1967)

\bibitem {Grae03}Gr\"{a}tzer, G.: General Lattice Theory. Birkh\"{a}user
Verlag, Berlin (2003)

\bibitem {Mar09}Marichal, J.-L.: Weighted lattice polynomials. Discrete
Mathematics 309(4), 814--820 (2009)

\bibitem {RG}Rico, A., Grabisch, M., Labreuche, Ch., Chateauneuf, A.:
Preference modeling on totally ordered sets by the Sugeno integral. Discrete
Applied Math. 147(1), 113--124 (2005)
\end{thebibliography}
\end{document}